\documentclass[letterpaper, 10pt,conference]{ieeeconf} 
\IEEEoverridecommandlockouts    
\newcommand{\eat}[1]{}
\usepackage{booktabs}
\usepackage [autostyle, english = american]{csquotes}
\MakeOuterQuote{"}

\usepackage{hyperref}
\hypersetup{
    colorlinks=true,
    linkcolor=blue,
    filecolor=black,      
    urlcolor=blue,
    }
\usepackage{subfig}
\usepackage{amsmath}
\usepackage{amstext}
\usepackage{amssymb}
\usepackage{amsfonts}

\usepackage{amsthm}  

\newtheorem{definition}{Definition}

\newtheorem{remark}{Remark}

\newtheorem{proposition}{Proposition}

\newtheorem{assumption}{Assumption}
\newtheorem{problem}{Problem}

\usepackage{todonotes}

\makeatletter

\newcommand{\Rmnum}[1]{\expandafter\@slowromancap\romannumeral #1@}

\usepackage{graphicx} 
\usepackage{cite}

\usepackage{mathrsfs}

\title{\LARGE \bf
Experimental Validation of a Real-Time Optimal Controller for Coordination of CAVs in a Multi-Lane Roundabout}

\author{Behdad Chalaki, \emph{IEEE Student Member}, Logan E. Beaver, \emph{IEEE Student Member},\\ Andreas A. Malikopoulos, \emph{IEEE Senior Member}%
\thanks{This research was supported in part by ARPAE’s NEXTCAR program under the award number DE- AR0000796 and by the Delaware Energy Institute (DEI). This support is gratefully acknowledged.}%
\thanks{The authors are with the Department of Mechanical Engineering, University of Delaware, Newark, DE 19716 USA (emails: \texttt{\{bchalaki;lebeaver;andreas\}@udel.edu}).} \thanks{Supplementary information and videos can be found at:
\newline \url{https://sites.google.com/view/ud-ids-lab/mlrb}}}

\begin{document}
\maketitle
\thispagestyle{empty}
\pagestyle{empty}
\begin{abstract}
Roundabouts in conjunction with other traffic scenarios, e.g., intersections, merging roadways, speed reduction zones, can induce congestion in a transportation network due to driver responses to various disturbances. Research efforts have shown that smoothing traffic flow and eliminating stop-and-go driving can both improve fuel efficiency of the vehicles and the throughput of a roundabout. In this paper, we validate an optimal control framework developed earlier in a multi-lane roundabout scenario using the University of Delaware's scaled smart city (UDSSC).
We first provide conditions where the solution is optimal. Then, we demonstrate the feasibility of the solution using experiments at UDSSC, and show that the optimal solution completely eliminates stop-and-go driving while preserving safety.
\end{abstract}

\indent



\section{Introduction} \label{sec:1}
Roundabouts generally provide better operational and safety characteristics over other types of intersections \cite{flannery1997operational,flannery1998safety,al2003dynamic,sisiopiku2001evaluation,mandavilli2008environmental}. However, the increase of traffic becomes a concern for roundabouts due to their geometry and priority system, even with moderate demands, some roundabouts may quickly reach capacity \cite{liu2013analysis,hummer2014potential,yang2004new}. Moreover, all incoming traffic may experience a significant delay if the circulating flow is heavy. Previous research has focused mainly on enhancing roundabout mobility and safety with improved metering, or traffic signal controls \cite{hummer2014potential,yang2004new,martin2016benefits,xu2016multi,martin2016capacity}. \eat{To investigate the potential of metering signals in improving roundabout operations during rush hour, Hummer et al. \cite{hummer2014potential} tested a metering approach for a signal-lane roundabout model and a two-lane roundabout model with different levels of approaching traffic demand. Martin-Gasulla et al. \cite{martin2016benefits} studied the benefits of metering signals for roundabouts with unbalanced flow patterns. Yang et al. \cite{yang2004new} proposed a traffic-signal control algorithm to eliminate the conflict points and weaving sections for multi-lane roundabouts by introducing a second stop line for left-turn traffic. Xu et al. \cite{xu2016multi} suggested a multi-level control system that combines metering signalization with full actuated control to serve different time periods throughout the day.} Zohdi and Rakha \cite{zohdy2013enhancing} showed that by using a cooperative adaptive cruise controller, they could improve the fuel efficiency and reduce the travel delay in a single lane roundabout compared to traditional roundabouts. 
As we move to increasingly complex emerging transportation systems, with changing landscapes enabled by connectivity and automation, future transportation networks could shift dramatically with the large-scale deployment of connected and automated vehicles (CAVs). 

Several efforts have been reported in the literature towards coordinating CAVs to reduce spatial and temporal speed variation of individual vehicles throughout the network. These variations can be introduced to the system through the environment, such as by breaking events, or due to the structure of the road network, e.g., intersections \cite{Dresner2008,Lee2012,azimi2014stip, Malikopoulos2017,hult2018optimal,zhang2019decentralized,bichiou2018developing} and cooperative merging \cite{Rios-Torres2017,Ntousakis2016aa, Zhao2018}.\eat{, and speed harmonization \cite{Malikopoulos2018c}.} One of the earliest efforts in this direction was proposed by Athans \cite{Athans1969} to efficiently and safely coordinate merging behaviors as a step to avoid congestion.
Since then, several research efforts have been reported in the literature proposing coordination of CAVs in traffic scenarios, such as merging roadways, urban intersections, and speed reduction zones. In earlier work, a decentralized optimal control framework was established for online coordination of CAVs in such traffic scenarios. The analytical solution, without considering state and control constraints, was presented in \cite{Rios-Torres2017, Ntousakis:2016aa} for online coordination of CAVs at merging roadways, a solution for two adjacent intersections was presented in  \cite{mahbub2019energy,chalaki2019a}, and coordination in single-lane roundabouts were investigated in \cite{Malikopoulos2018a}. \eat{The solution of the optimal control problem which considered state and control constraints was presented in \cite{Malikopoulos2017} at an urban intersection without considering a rear-end collision avoidance constraint. The conditions under which the rear-end collision avoidance constraint never becomes active are discussed in \cite{Malikopoulos2018c}.} A thorough review of the state-of-the-art methods and challenges of CAV coordination is provided in \cite{Malikopoulos2016a,guanetti2018control}.

Previously, we have experimentally validated the solution of the unconstrained problem using 10 robotic CAVs at the University of Delaware's scaled smart city (UDSSC) considering a merging roadway scenario \cite{Malikopoulos2018b} and a corridor \cite{Beaver2020DemonstrationCity}.
Other efforts \cite{berntorp2019motion, Hyldmar2019} have demonstrated CAV maneuvers in a scaled environment either by utilizing $2-3$ CAVs \cite{berntorp2019motion} or by focusing on highway driving conditions \cite{Hyldmar2019}.

In this paper, we validate a decentralized optimal control framework developed earlier \cite{Malikopoulos2019b, Malikopoulos2020} in a multi-lane roundabout scenario using UDSSC. Unlike previous work, we guarantee that all safety, state, and control constraints are satisfied by our analytical solution. We demonstrate the feasibility of our framework using experiments in UDSSC. In particular, we implement the solution in real time for $9$ CAVs in a multi-lane roundabout with $3$ areas for the potential lateral collision. 
Finally, we verify that our optimal solution completely eliminates stop-and-go driving while preserving safety.

The remainder of the paper is organized as follows. 
In Section \Rmnum{2}, we introduce our modeling framework. In Section \Rmnum{3}, we provide the analytical solution of the optimal control problem. Then, we present experimental validation and results in Section \Rmnum{4}, and finally, we draw concluding remarks and discuss future work in Section \Rmnum{6}.

\section{Problem Formulation} \label{sec:2}

\subsection{The Roundabout Scenario}
\label{sec:2a}

In this paper, we consider a multi-lane roundabout with three CAV inflows and three areas where lateral collisions between CAVs may occur shown in Fig. \ref{fig:roundabout}.
However, our proposed solution does not depend on the specific paths presented in this work and can be applied in any scenario in a roundabout.
To navigate the roundabout, we define a \emph{control zone}, which starts upstream from the roundabout and ends at each roundabout exit (Fig. \ref{fig:roundabout}). The control zone has an associated \emph{coordinator}, which stores information about the geometry of the roundabout and the trajectory information of each CAV in the control zone. The coordinator does not make any decisions and only acts as a database.
\begin{figure}[ht]
    \centering
    \includegraphics[width=0.65\linewidth]{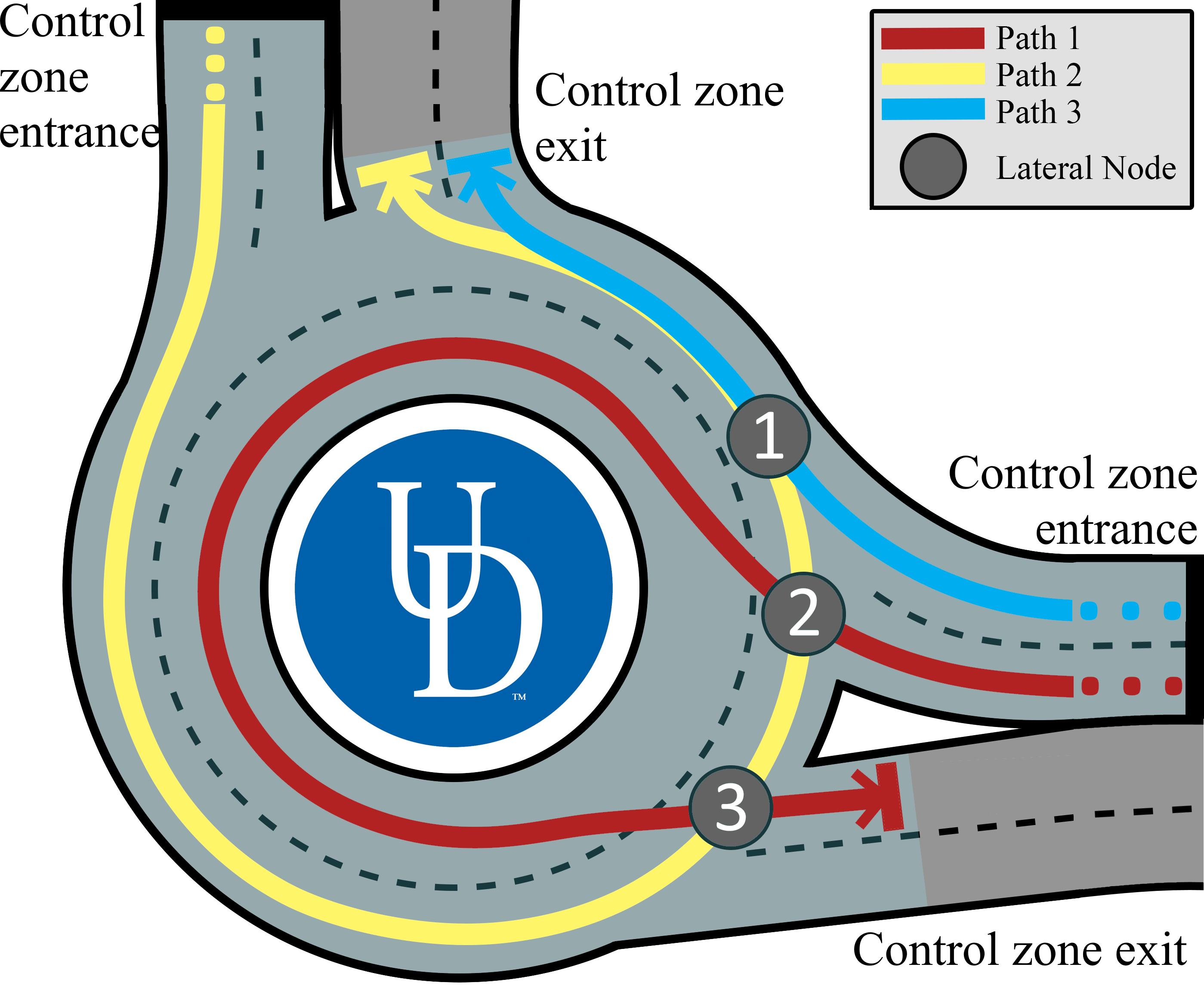}
    \caption{A schematic of the roundabout scenario. The highlighted control zone continues upstream from the roundabout. }
    \label{fig:roundabout}
\end{figure}

Let $\mathcal{Q}(t)\subset\mathbb{N} $ be the set of CAVs at time $t\in\mathbb{R}^{+}$ which are inside the control zone. Upon entering the control zone at time $t_i^0\in\mathbb{R}^{+}$, CAV $i\in\mathcal{Q}(t)$ retrieves the trajectory information of every other CAV $j\in\mathcal{Q}(t) \setminus \{i\}$ and generates an energy-optimal safe trajectory through the control zone. Then, 
CAV $i$ broadcasts its trajectory to the coordinator. Finally, when CAV $i$ exits the control zone at time $t_i^f$, it is removed from the set $\mathcal{Q}(t)$.
A detailed discussion about the communication of the coordinator with the CAVs is presented in \cite{Beaver2020DemonstrationCity}.

\begin{definition}
We define each lateral node with a unique index, $n\in \{1,2,3\}$,  at each area inside the roundabout where there might be a potential lateral collision (Fig. \ref{fig:roundabout}). 
\end{definition}
\begin{definition}
For each CAV $i\in\mathcal{Q}(t)$, we define 
$\mathcal{N}_i$ as the set of all lateral nodes on the path of CAV $i$.
\end{definition}
For instance, if CAV $i$ is traveling along the path $1$ (Fig. \ref{fig:roundabout}), the set of collision nodes is $\mathcal{N}_i =\{2,3\}$.

\begin{definition}
For each CAV $i\in\mathcal{Q}(t)$, upon entering the control zone at time $t_i^0$,  we define the set of lateral nodes shared with each CAV $j\in\mathcal{Q}(t_i^0) \setminus \{i\}$ as,
\begin{gather}\label{1a}
\mathcal{C}_{i,j}=\big\{n ~|\; n\in\:\mathcal{N}_i\cap\mathcal{N}_j \big\}.
\end{gather}
\end{definition}
\subsection{Vehicle Model and Constraints}
\eat{We represent the dynamics of each CAV
$i\in\mathcal{Q}(t)$, 
with the state equation and initial condition
\begin{equation}
\dot{\mathbf{x}}_{i}(t)=f(t,\mathbf{x}_{i}(t),u_{i}(t)),\qquad \mathbf{x}_{i}(t_{i}^{0})=\mathbf{x}_{i}^{0},\label{eq:model}\\
\end{equation}
where $t\in\mathbb{R}^{+}$, and $\mathbf{x}_{i}(t)=\left[p_{i}(t) , v_{i}(t)\right]  ^{T}$, $u_{i}(t)$ are the state and control input of $i$ at time $t$.
Let $t_{i}^{0}$ be the time that $i$
enters the control zone, and $\mathbf{x}_{i}^{0}=\left[p_{i}^{0} , v_{i}^{0} \right] ^{T}$ be the state at this time. 
}
We model dynamics of each CAV $i\in\mathcal{Q}(t)$ as a double integrator 
\begin{equation}%
\begin{split}
\dot{p}_{i} &  =v_{i}(t),\\
\dot{v}_{i} &  =u_{i}(t),
\label{eq:model2}
\end{split}
\end{equation}
where $p_{i}(t)\in\mathcal{P}_{i}$, $v_{i}(t)\in\mathcal{V}_{i}$, and
$u_{i}(t)\in\mathcal{U}_{i}$ denote the position, speed and
acceleration/deceleration (control input) of each CAV $i$ inside the control zone. The sets $\mathcal{P}_{i}$,
$\mathcal{V}_{i}$, and $\mathcal{U}_{i}$
are complete and totally bounded subsets of $\mathbb{R}$.
\eat{
Let
$x_{i}(t)=\left[p_{i}(t) ~ v_{i}(t)\right]  ^{T}$ denote the state of each CAV $i$, with initial value
$x_{i}^{0}=\left[p_{i}^{0} ~ v_{i}^{0} \right]  ^{T}$, where $p_{i}^{0}= p_{i}(t_{i}^{0})=0$ at the entry of the control zone, taking values in $\mathcal{X}_{i}%
=\mathcal{P}_{i}\times\mathcal{V}_{i}$.  The state space 
$\mathcal{X}_{i}$ for each CAV $i$ is
closed with respect to the induced topology on $\mathcal{P}_{i}\times
\mathcal{V}_{i}$ and thus, it is compact.
We need to ensure that for any initial state $(t_i^0, x_i^0)$ and every admissible control $u(t)$, the system \eqref{eq:model} has a unique solution $x(t)$ on some interval $[t_i^0, t_i^f]$, where $t_i^f$ is the time that CAV $i\in\mathcal{N}(t)$ exits the \textit{control zone}. 
The following observations from \eqref{eq:model} satisfy some regularity conditions required both on $f$ and admissible controls $u(t)$ to guarantee local existence and uniqueness of solutions for \eqref{eq:model}: a) The function $f$ is continuous in $u$ and continuously differentiable in the state $x$, b) The first derivative of $f$ in $x$, $f_x$, is continuous in $u$, and c) The admissible control $u(t)$ is continuous with respect to $t$.}

For each CAV $i\in\mathcal{Q}(t)$ the control input and speed at time $t\in\lbrack t_{i}%
^{0},t_{i}^{f}]$ are bounded by
\begin{equation}%
\begin{split}
u_{\min} &  \leq u_{i}(t)\leq u_{\max},\\
0   < v_{\min} & \leq v_{i}(t)\leq v_{\max},
\end{split}
\label{speed_accel constraints}%
\end{equation}
where $u_{\min}$, $u_{\max}$ are the minimum deceleration and maximum
acceleration, and $v_{\min}$, $v_{\max}$ are the minimum and maximum speed limits respectively. \eat{For simplicity, we do not consider CAV diversity. Thus, in the rest of the paper, we set $u_{i,\min}=u_{\min}$ and $u_{i,\max}=u_{\max}.$}
\begin{definition}
For any CAV $i\in\mathcal{Q}(t)$, if there exists a CAV $k\in\mathcal{Q}(t)$ which leads CAV $i$, we define $d_i(t)$ as the bumper-to-bumper distance from CAV $k$ to CAV $i$. If no such CAV $k$ leads CAV $i$, then we let $d_i(t)\to\infty$.
\end{definition}
To 
guarantee no rear-end collision occurs between CAV $i\in\mathcal{Q}(t)$ and the preceding CAV $k\in\mathcal{Q}(t)$, we impose the following rear-end safety constraint,
\begin{equation}\label{eq:rearEndSafety}
  d_i(t)\geq \delta_i(t),
\end{equation}
where $\delta_i(t)$ is the safe distance that depends on CAV's speed,
\begin{equation}
    \delta_i(t) = \gamma + \varphi\, v_i(t),
\end{equation}
where $\gamma,\varphi\in\mathbb{R}$ are the standstill distance and reaction time, respectively.
\eat{
\begin{definition}
 For each CAV $i\in\mathbb{N}$, we define inverse of position as $p_i^{-1}: \mathbb{R}^+\rightarrow\mathcal{P}_{i}. $
\end{definition}}

For each CAV $i\in\mathcal{Q}(t),$ the distance to a node $n\in\mathcal{N}_i$ is denoted by the function $l_i: \mathcal{N}_i\rightarrow\mathcal{P}_{i}$.
To guarantee lateral collision avoidance, we impose the following time headway constraint for every CAV  $i\in\mathcal{Q}(t)$,
\begin{align} \label{eq:lateralSafety}
    &\big|p_i^{-1}(l_i(n)) - p_j^{-1}(l_j(n))\big| \geq t_h,\\
    &\forall\, n \in \mathcal{C}_{i,j}, ~~ \forall j\in\mathcal{Q}(t)\setminus\{i\},\nonumber
\end{align}
where $t_h\in\mathbb{R}^{+}$ is the minimum time headway between any two vehicles entering node $n$.
Note that as position is strictly-increasing for all $t\in[t_i^0, t_i^f]$, the inverse position \eqref{eq:lateralSafety} has a closed-form representation \cite{Malikopoulos2019b,Malikopoulos2020}. 


\begin{remark}
The lateral safety constraint \eqref{eq:lateralSafety} can relax the first-in-first-out coordination policy for CAVs $i,j\in\mathcal{Q}(t)$,
which is common in the literature.
\end{remark}

Next, we formulate a decentralized optimal control problem for each CAV $i\in\mathcal{Q}(t)$ in order to minimize their energy consumption over the interval $t\in[t_i^0, t_i^f]$.
\begin{problem} \label{prb:optimalControl} When a CAV $i\in\mathcal{Q}(t)$ enters the control zone, it solves the following optimal control problem:
\begin{equation}
\label{eq:decentral}
\begin{array}{ll}
\min\limits_{u_i(t)\in \mathcal{U}_i} &\frac{1}{2}\int_{t^0_i}^{t^f_i} u^2_i(t)~dt,\\ 
\emph{subject to}%
:&\eqref{eq:model2},\eqref{speed_accel constraints},\eqref{eq:rearEndSafety},\eqref{eq:lateralSafety},\nonumber\\
\emph{given } &p_{i}(t_{i}^{0}),v_{i}(t_{i}^{0})\text{, } p_{i}(t_{i}^{f}).\nonumber
\end{array}
\end{equation}
\end{problem}

\eat{
To derive the analytical solution of Problem \ref{prb:optimalControl}, we follow the standard methodology used in optimal control problems with 
state and control constraints \cite{bryson1975applied,ross2015primer,Malikopoulos2017}. First, we start with the unconstrained solution to Problem \ref{prb:optimalControl}. If this solution violates any of the state or control constraints, then it is pieced together with the arc corresponding to the violated constraint. This yields a set of algebraic equations that are solved simultaneously using the boundary conditions of Problem \ref{prb:optimalControl} and interior conditions between the arcs. If the resulting solution, which incorporates the optimal switching time between arcs, violates another constraint, then the two arcs are pieced together with a third, which corresponds to the new constraint which has been violated. This yields a larger set of algebraic equations that must be solved simultaneously using the boundary conditions of Problem \ref{prb:optimalControl} and the interior 
conditions between arcs, including the optimal time and state to switch between them. This process is repeated until a piecewise-continuous state trajectory is found, which is a feasible solution for Problem \ref{prb:optimalControl}.
}

To derive the analytical solution to Problem \ref{prb:optimalControl} we may follow the standard methodology used in optimal control problems with state and control constraints \cite{bryson1975applied,ross2015primer,Malikopoulos2017}. This is a recursive process which grows in computational complexity with the number of constraints that may become active in the system.
In general, there is no analytical expression for the solution of Problem \ref{prb:optimalControl} when a safety constrained arc becomes active. Additionally, finding a piecewise-continuous state trajectory that optimally pieces the different arcs together is very computationally challenging, and it may become prohibitively difficult for a CAV to solve the optimal control problem onboard in real time \cite{xiao2019decentralized}.


In our approach, we consider the unconstrained solution to Problem \ref{prb:optimalControl}.
\eat{, i.e., the solution when the state, control, and safety constraints are neglected.}
This has the advantage of guaranteed energy-optimality while being significantly easier to implement on a real vehicle. The optimal unconstrained trajectory in this case can be found by Hamiltonian analysis \cite{Malikopoulos2017}, 
\begin{align}
    p_i(t) =&~ a_i t^3 + b_i t^2 + c_i t + d_i,\label{eq:pUnc}\\
    v_i(t) =&~ 3a_i t^2 + 2 b_i t + c_i,\label{eq:vUnc} \\
    u_i(t) =&~ 6a_i t + 2 b_i,\label{eq:uUnc}
\end{align}
with the boundary conditions
\begin{align}
    p_i(t_i^0) = p_i^0, ~ v_i(t_i^0) = v_i^0, ~
    p_i(t_i^f) = p_i^f, ~ u_i(t_i^f) = 0, \label{eq:optimalBC}
\end{align}
where $u_i(t_i^f) = 0$ results from the velocity being unspecified at $t_i^f$ (the transversality condition). 

To derive an energy-optimal control input that can be computed in real time, we seek to minimize the exit time of CAV $i\in\mathcal{Q}(t)$ from the control zone and impose \eqref{eq:pUnc} - \eqref{eq:uUnc} as an \emph{energy-optimal motion primitive}. This results in a new energy and time-optimal scheduling problem \cite{Malikopoulos2019b}.

\begin{problem} \label{prb:schedulingProblem}
When a CAV $i\in\mathcal{Q}(t)$ enters the control zone, it derives its minimum travel time such that the resulting trajectory is unconstrained and does not violate any state, control, or safety constraints.
\begin{equation}
\begin{array}{ll}
\min\limits_{_{a_i, b_i, c_i, d_i}} &t_i^f, \\\nonumber
\emph{subject to: } &\eqref{speed_accel constraints}, \eqref{eq:rearEndSafety}, \eqref{eq:lateralSafety}, \\\nonumber
&\eqref{eq:pUnc}, \eqref{eq:vUnc}, \eqref{eq:uUnc}, 
\eqref{eq:optimalBC}. \\
\end{array}
\end{equation}
\end{problem}
The solution of Problem \ref{prb:schedulingProblem} yields the minimum $t_i^f$ such that the generated trajectory is the unconstrained optimal solution to Problem \ref{prb:optimalControl}. 
Next, we provide the assumptions we imposed in our approach on each CAV $i\in\mathcal{Q}(t)$.

\begin{assumption}\label{smp:noError}
There are no errors or delays in the vehicle-to-vehicle and vehicle-to-infrastructure communication. 
\end{assumption}

\begin{assumption}\label{smp:tracking}
   Vehicle-level control is handled by a low-level controller which can perfectly track the trajectory generated by solving Problem \ref{prb:schedulingProblem}.
\end{assumption}

The first assumption ensures that we address the deterministic case. It is relatively straightforward to relax this assumption as long as the noise or delays are bounded. The second assumption is to decouple the motion planning and vehicle control, which makes the problem tractable. By tuning the low-level controller, it can be ensured that the prescribed trajectory is followed.

\section{Analytical solution} \label{sec:3}

We seek to transform Problem \ref{prb:schedulingProblem} into an equivalent formulation which can be solved in real time. First, without loss of generality, we consider the domain of Problem \ref{prb:schedulingProblem} to be $t \in [0, t_i^f]$ and $p_i(t)\in[0, S_i]$, where $S_i$ is the length of the control zone corresponding to CAV $i$'s path. This results in a new set of boundary conditions,
\begin{align}
    p_i(t_i^0 = 0) &= 0, &v_i(t_i^0 = 0) &= v_i^0, \label{eq:ICs}\\
    p_i(t_i^f) &= S_i, &u_i(t_i^f) &= 0.  \label{eq:BCs}
\end{align}

Next, we substitute \eqref{eq:ICs} into \eqref{eq:pUnc} and \eqref{eq:vUnc} yielding
\begin{align}
    p_i(t_i^0 = 0) &= d_i = 0, \label{eq:di}\\
    v_i(t_i^0 = 0) &= c_i = v_i^0, \label{eq:ci}
\end{align}
which must always hold for Problem \ref{prb:schedulingProblem}. Next, we substitute \eqref{eq:BCs} into \eqref{eq:uUnc}, which yields,
    $u_i(t_i^f) = 6a_i t_i^f + 2b_i = 0$.
This implies that
\begin{equation} \label{eq:ai}
    a_i = -\frac{b_i}{3t_i^f}.
\end{equation}
Next, we substitute \eqref{eq:di}-\eqref{eq:ai} into \eqref{eq:pUnc} yielding
\begin{equation} \label{eq:pTf}
    p_i(t_i^f) = -\frac{b_i}{3 t_i^f} {t_i^f}^3 + b_i {t_i^f}^2 + v_i^0 t_i^f = S_i.
\end{equation}
Hence, equation \eqref{eq:pTf} simplifies to
\begin{equation} \label{eq:bi}
    b_i = \frac{3(S_i - v_i^0 t_i^f)}{2 {t_i^f}^2}.
\end{equation}

We may further simplify Problem \ref{prb:schedulingProblem} by finding a compact domain of feasible $t_i^f$ by explicitly applying the speed and control constraints. Let $t_{i,\min}^f$ and $t_{i,\max}^f$ denote the lower bound and upper bound on $t_i^f$ respectively, which is imposed by the state and control constraints.

\begin{proposition} \label{prp:lowerBound}
For each CAV $i\in\mathcal{Q}(t)$, the lower bound on exit time of the control zone, $t_{i,\min}^f$, is computed as follows 
 \begin{equation}
     t_{i,\min}^f = \min \{ t_{i,u_{\max}}^f  , t_{i,v_{\max}}^f  \},
 \end{equation}
 where 
 \begin{equation}
    t_{i,u_{\max}}^f = \frac{\sqrt{9 v_0^2 + 12 S_i u_{\max}} - 3 v_0}{2 u_{\max}},
\end{equation}
\begin{equation}
     t_{i,v_{\max}}^f = \frac{3 S_i}{v_i^0 + 2 v_{i,\max}}.
 \end{equation}
\end{proposition}
\begin{proof}
There are two cases to consider:
\textbf{Case $1$:} CAV $i$ achieves its maximum control input at entry of the control zone, as $u_i(t_i^0)=u_{\max}$. \textbf{Case $2$:} CAV $i$ achieves its maximum speed at the end of control zone, as $v_i(t)$ is strictly increasing $v_i(t_i^f)=v_{\max}$.

In case $1$, by \eqref{eq:uUnc}, we have 
\begin{equation}\label{eq:tfumax}
    u_i(t_i^0 = 0) = 2b_i = u_{\max} .
\end{equation}
Substituting \eqref{eq:bi} into \eqref{eq:tfumax} and solving for $t_i^f$, yields the quadratic equation
\begin{equation}\label{eq:tfumax2}
    u_{\max}{t_i^f}^2 + 3v_i^0 t_i^f - 3S_i = 0,
\end{equation}
which has two real roots with opposite signs, as $t_{i,1}^f t_{i,2}^f = \frac{-3 S_i}{u_{\max}}<0$.
Thus, $t_{i,u_{\max}}^f>0$ is computed by
\begin{equation}
    t_{i,u_{\max}}^f = \frac{\sqrt{9 v_0^2 + 12 S_i u_{\max}} - 3 v_0}{2 u_{\max}}.
\end{equation}

For case $2$, by \eqref{eq:vUnc}, we have
\begin{equation} \label{eq:vminPoly}
    v_i(t_i^f) = \frac{a_i}{3} {t_i^f}^2 + \frac{b_i}{2} t_i^f + v_i^0 = v_{\max}.
\end{equation}
Substituting \eqref{eq:ai} and \eqref{eq:bi} into \ref{eq:vminPoly} yields
\begin{align}
     v_i(t_i^f) &= 3 \Big(\frac{-b_i}{3 t_i^f}\Big) {t_i^f}^2 + 2 b_i t_i^f + v_0 \\
     &= b_i t_i^f + v_0 = \frac{3(S_i - v_i^0 t_i^f)}{2 t_i^f} + v_0 = v_{\max},\nonumber
\end{align}
which simplifies to
\begin{equation}
     t_{i,v_{\max}}^f = \frac{3 S_i}{v_i^0 + 2 v_{\max}} .
 \end{equation}
 Thus, our lower bound on $t_i^f$ is given by
 \begin{equation}
     t_{i,\min}^f = \min \{ t_{i,u_{\max}}^f  , t_{i,v_{\max}}^f  \}.
 \end{equation}
\end{proof}

\begin{proposition} \label{prp:upperBound}
For each CAV $i\in\mathcal{Q}(t)$, the upper bound on exit time of the control zone, $t_{i,\max}^f$, is computed as follows 
\begin{equation}
        t_{i,\max}^f = 
        \begin{cases}
            t_{i,v_{\min}},&\text{ if }\ 9 {v_i^0}^2 + 12 S_i u_{i,\min} < 0,\\
            \max \{ t_{i,u_{\min}}^f  , t_{i,v_{\min}}^f\} ,&\text{otherwise.}\ \\
        \end{cases}
\end{equation}
 where 
\begin{align}
        t_{i,v_{\min}} = \frac{3 S_i}{v_i^0 + 2 v_{\min} },
        t_{i,u_{\min}} = \frac{\sqrt{9 v_0^2 + 12 S_i u_{\min}} -3 v_0 }{2 u_{\min} }.
\end{align}
\end{proposition}
\begin{proof}
Similar steps to Proposition \ref{prp:lowerBound} can be followed to find the upper bound for $t_{i}^f$. 
Note that when $9 v_0^2 + 12 S_i u_{\min} < 0$ there is no real value of $t_i^f$ which satisfies all of the boundary conditions simultaneously. 
In this case the lower bound is given by the $v_{\min}$ case.
The proof for the $v_{\min}$ case is identical to Proposition \ref{prp:upperBound} and is omitted. 
\end{proof}{}

Finally, we may write an equivalent formulation of Problem \ref{prb:schedulingProblem}, which optimizes a single variable, $t_i^f$ over a compact set $[t_{i,\min}^f,t_{i,\max}^f]$.

\begin{problem} \label{prb:final}
When CAV $i\in\mathcal{Q}(t)$ enters the control zone it derives the minimum exit time such that the resulting unconstrained trajectory does not violate any safety constraints.
\begin{equation}
\begin{array}{ll}
&\min\limits_{t_i^f} t_i^f, \\\nonumber
\emph{subject to: } &\eqref{eq:rearEndSafety}, \eqref{eq:lateralSafety}, \\\nonumber
& \eqref{eq:di}, \eqref{eq:ci}, \eqref{eq:ai}, \eqref{eq:bi},\\\nonumber
&t_i^f \in [t_{i,\min}^f, t_{i,\max}^f].
\end{array}
\end{equation}
\end{problem}

\eat{\textcolor{blue}{Problem \ref{prb:final} is formulated for each path and is only dependent on path length (is independent of road geometry). It may be formulated with an arbitrary number of nodes in \eqref{eq:lateralSafety}.}}
Problem \ref{prb:final} can then be numerically solved in real time by each CAV $i\in\mathcal{Q}(t)$ upon entering the control zone.
In the next section, we describe our experimental testbed and the implementation of the proposed approach.

\section{Experimental validation and results}
To experimentally validate our method, experiments were carried out in UDSSC, 
using nine scaled CAVs traveling along three different conflicting paths (Fig. \ref{fig:roundabout}). 
UDSSC is a 1:25 scale testbed designed to replicate real-world traffic scenarios and test cutting-edge control technologies in a safe and scaled environment (see \cite{Beaver2020DemonstrationCity} for more details).
Upon entering the control zone, CAV $i\in\mathcal{Q}(t)$ determines its trajectory by solving Problem \ref{prb:final} numerically. CAV $i$, then, follows this fixed trajectory until it exits the control zone.


During the experiment we used the following parameters in Problem \ref{prb:final}: $v_{\max} = 0.15$ m/s, $v_{\min} = 0.05$ m/s, $u_{\max} = 0.45$ m/s$^2$, $u_{\min} = -0.45$ m/s$^2$, and $t_h = 1.0$ s.
We repeated the experiment five times to collect multiple data sets and to give an estimate on the noise and disturbances present in the system. The CAV inflows were explicitly configured to lead to lateral collisions in the uncontrolled case. Next, we present and discuss the results of these experiments. 


To validate our proposed controller in UDSSC, several pieces of data were collected throughout the five experiments. First, the position, speed\eat{, location} within the UDSSC, and a timestamp for each CAV was streamed back to the mainframe at a rate of $20$ Hz. Furthermore, the state and time of each CAV entering the control zone were recorded, as well as the computed and achieved exit time. These results are summarized in Table \ref{tab:results}. Note that the minimum speed of any
CAVs at $1:25$ scaled testbed across all five experiments is $0.12$ m/s  ($7$ mph at full scale), which demonstrates that stop and go driving has been completely eliminated. Additionally, the average CAV speed is $0.42$ m/s ($24$ mph at full scale), which implies that most CAVs are traveling near $v_{\max}$ and must apply minimal control effort.

\begin{table}[ht]
    \caption{Average velocity and travel time results for the 5 experiments. RMSE is normalized by travel time for each CAV.}
    \centering
    \begin{tabular}{c|ccc}
        Experiment & $v_{\min}$ [m/s] & $v_{\text{avg}}$ [m/s] & Travel Time RMSE \\
        \toprule
        1 & 0.16 & 0.41 & 2.71\,\% \\
        2 & 0.27 & 0.45 & 1.54\,\% \\
        3 & 0.18 & 0.41 & 4.03\,\% \\
        4 & 0.12 & 0.43 & 1.92\,\% \\
        5 & 0.21 & 0.42 & 1.38\,\%
    \end{tabular}
    \label{tab:results}
\end{table}

The exit time data for each CAV is visualized in Fig. \ref{fig:tifResult}, where the grey bars represent the feasible space of $t_i^f$, the wide black bars correspond with the solution of Problem \ref{prb:final}, and the thin red bars show the achieved exit time for each CAV. From Table \ref{tab:results}, the error between desired and actual exit time varies between $2-4$\%. This error comes from the CAV's ability to track the desired trajectory and shows that Assumption \ref{smp:tracking} is reasonable for well-tuned CAVs in UDSSC.

\begin{figure}[ht]
    \centering
    \includegraphics[width=0.9\linewidth]{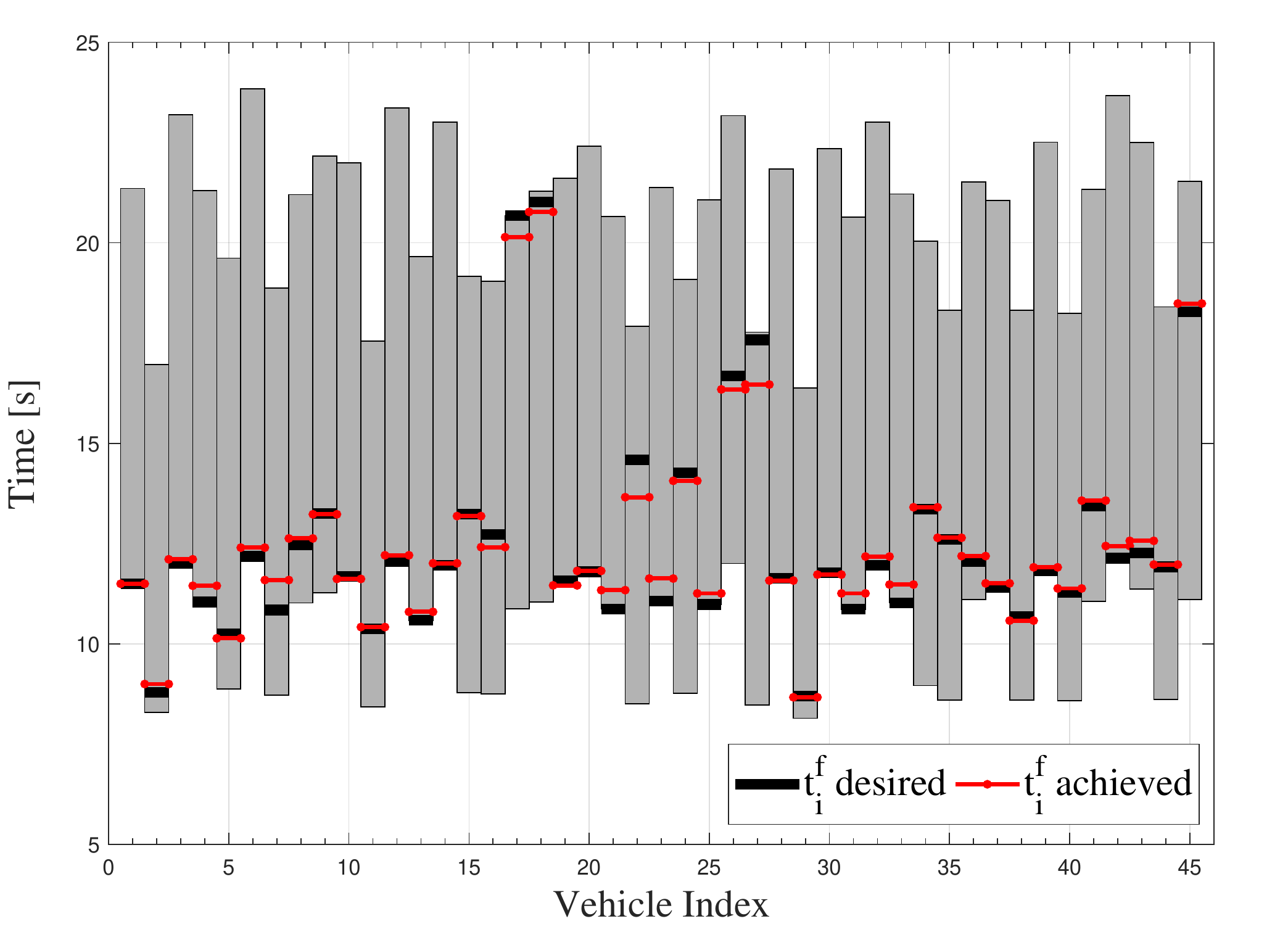}
    \caption{Estimated and actual arrival time for each vehicle over all experiments.} 
    \label{fig:tifResult}
\end{figure}

The position trajectory of an ego-CAV following path $2$ is given in Fig. \ref{fig:trajectory}. The ego-CAV's position is denoted by the dashed red line, while the positions of two other CAVs are represented by dotted black lines. The lateral collision constraints are denoted by vertical black bars, and the rear-end safety constraint is the hashed region on the graph. There are two other CAVs shown; one is on Path 3 and merges in front of the ego-CAV at collision node $1$ (Fig. \ref{fig:roundabout}) and the second CAV leads the ego-CAV on path $2$.

Figure \ref{fig:trajectory} demonstrates that, in reality, Assumption \ref{smp:noError} is too strong. The trajectory generated by the ego-vehicle may not violate any constraints, but the actual trajectory violates the rear-end safety constraint by a car length ($0.2$ m). However, at this speed, the rear-end safety constraint requires a three-car length gap, so a robust control formulation of Problem \ref{prb:final} could likely guarantee collision avoidance. This can also be seen in the lateral collision avoidance constraint in Fig. \ref{fig:trajectory}, where a later CAV crosses node $3$ in a way that violates the time headway constraint (again, without leading to an actual collision).

\begin{figure}[ht]
    \centering
    \includegraphics[width=0.9\linewidth]{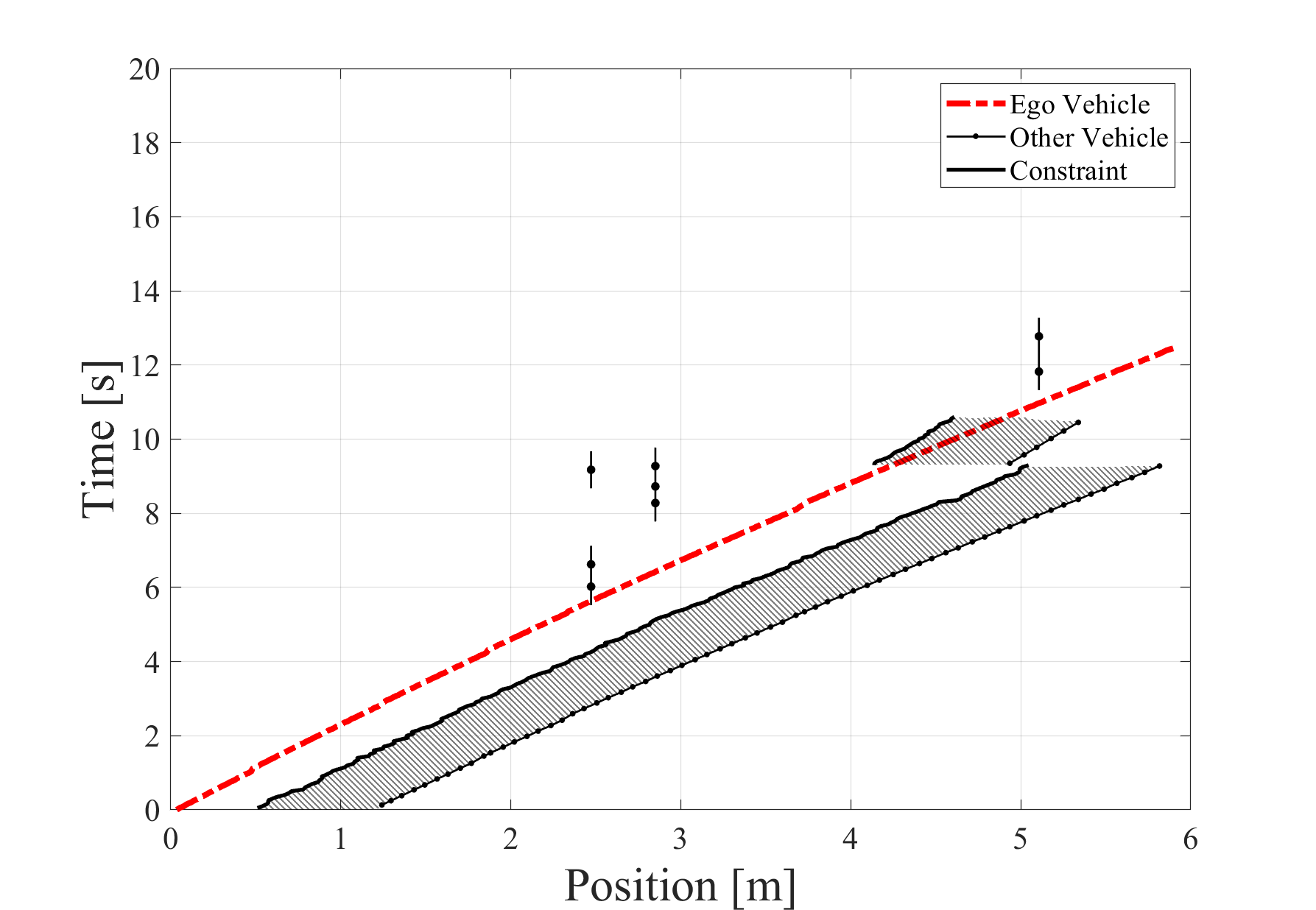}
    \caption{Position trajectory for the third vehicle entering from path $2$ in the $5$th experiment. The lateral constraints are shown as vertical lines, and the rear-end safety constraint is the hashed region.}
    \label{fig:trajectory}
\end{figure}

Finally, the average, maximum, and minimum speed for each CAV across all experiments are given in Fig. \ref{fig:trajectoryBlock}. Each figure corresponds to a single path (see Fig. \ref{fig:roundabout}) and is taken over $15$ CAV ($3$ CAVs per path over five experiments). The CAVs' positions are taken directly from VICON and numerically derived using a first-order method.

From Fig. \ref{fig:trajectoryBlock}, the average speed for CAVs on each path is very close to constant. Path $1$ shows the most variance, which is due to the distance between collision nodes $2$ and $3$ on path $1$ (see Fig. \ref{fig:roundabout}). In order for a CAV $i\in\mathcal{Q}(t)$ which is traveling along path $1$ to reduce its arrival time at node $2$, it must make a proportionally larger reduction in the value of $t_i^f$. This is a side effect of enforcing the unconstrained trajectory on each CAV over the entire control zone.
Additionally, the entrance to the control zone along path $3$ follows a sharp right turn. This results in a relatively lower average trajectories in Fig. \ref{fig:trajectoryBlock}\subref{c}, as the dynamics of the CAVs reduce their speed while rounding these turns, causing them to enter the control zone at a lower speed.
Finally, there are instances in Fig. \ref{fig:trajectoryBlock}\subref{b} where the maximum vehicle speed surpasses the speed limit. This is a result of stochasticity in the vehicle dynamics and sensing equipment, as well as environmental disturbances, on our deterministic controller.

\eat{
\begin{figure}[ht]
    \centering
    \includegraphics[width=0.99\linewidth]{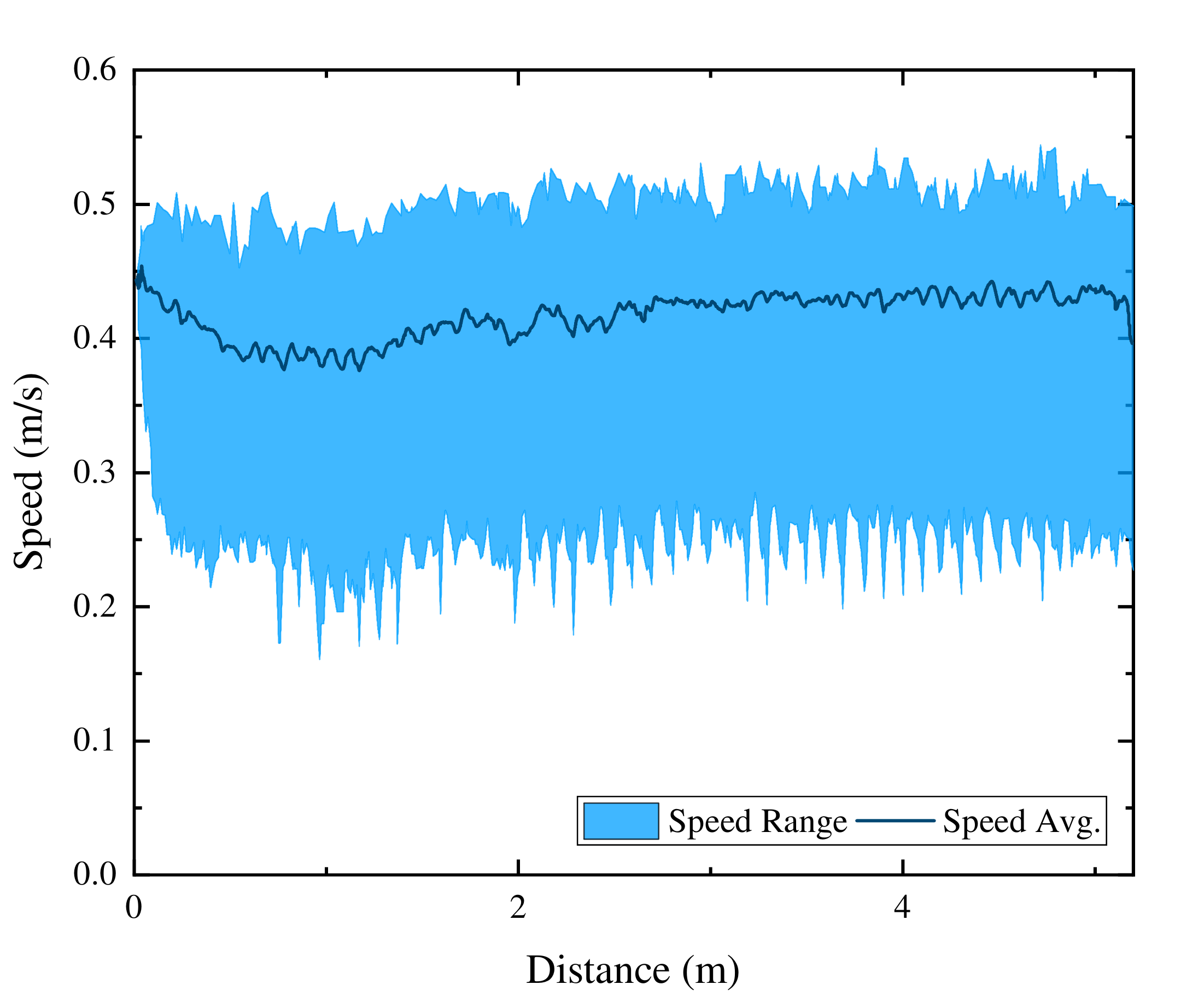}
    \caption{ Maximum, minimum, and average speed of CAVs on the path 1 (Fig.\ref{fig:roundabout}) over the control zone length for the five experiments.}
    \label{fig:trajectory1}
\end{figure}

\begin{figure}[ht]
    \centering
    \includegraphics[width=0.99\linewidth]{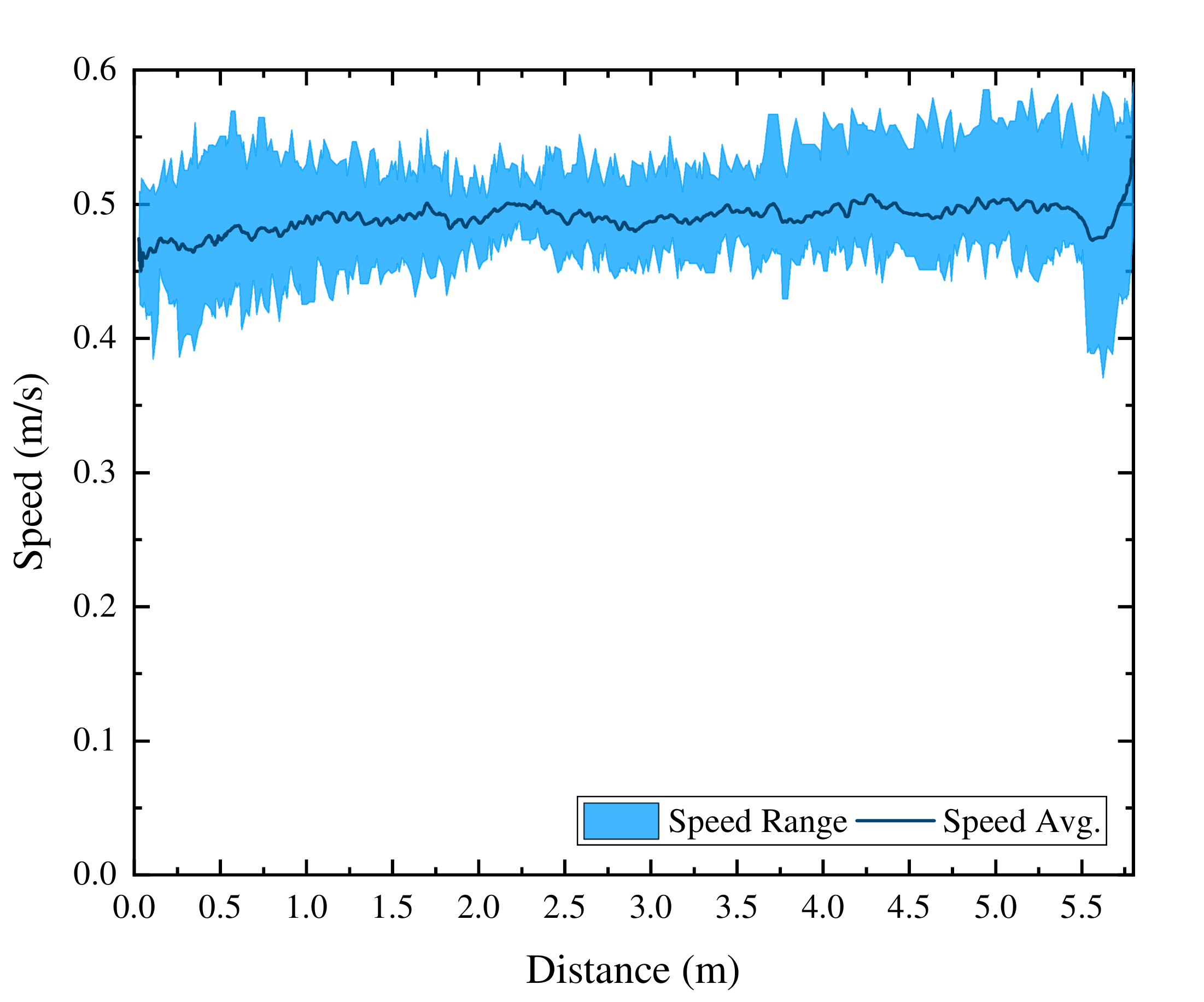}
    \caption{Maximum, minimum, and average speed of CAVs on the path 2 (Fig.\ref{fig:roundabout}) over the control zone length for the five experiments.}
    \label{fig:trajectory2}
\end{figure}
}

\begin{figure*}[ht]
    \centering
   \subfloat[][]{\includegraphics[width=0.33\linewidth]{figures/avgSpeed_eastInner.pdf}\label{a}}
     \subfloat[][]{\includegraphics[width=0.33\linewidth]{figures/avgSpeed_north.pdf}\label{b}}
     \subfloat[][]{\includegraphics[width=0.33\linewidth]{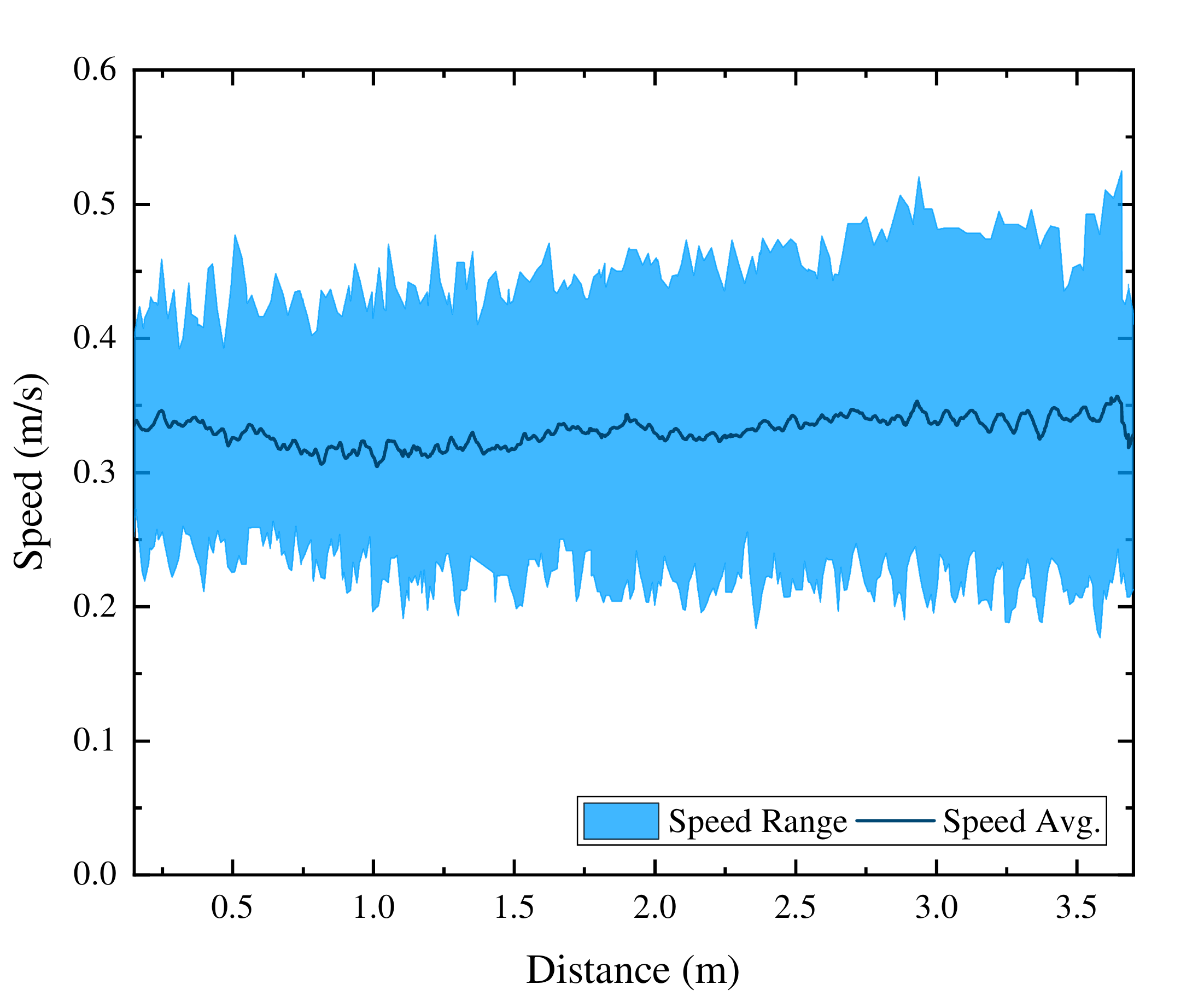}\label{c}}
    \caption{Speed range and average for all CAVs on \protect\subref{a} path $1$, \protect\subref{b} path $2$, and \protect\subref{c} path $3$ across all experiments.}
    \label{fig:trajectoryBlock}
\end{figure*}

\section{Conclusion}

In this paper, we validated experimentally a decentralized optimal control framework, developed earlier \cite{Malikopoulos2019b, Malikopoulos2020}, in a traffic scenario that included coordination of CAVs in a multi-lane roundabout. We demonstrated that the framework can be implemented in real time when multiple locations for potential lateral collisions exist. In our experiment, we used $9$ CAVs over a series of $5$ experiments and showed that the CAVs can cross the roundabout without stop-and-go driving while avoiding collisions. The next step is to enhance the problem formulation to account for noise, disturbances, communication delays, and low-level tracking error in the CAVs. Exploring the trade-off between the state and control constraints (increasing the feasible space of $t_i^f$) and the safe lateral time headway (decreasing the feasible space of $t_i^f$) is another direction for future research.


\bibliographystyle{IEEEtran}
\bibliography{Bib/main}

%

%

\end{document}